\documentclass[10pt]{article}

\font\smallit=cmti10

\usepackage{amssymb,latexsym,amsmath,epsfig,amsthm,enumitem}

\makeatletter

\renewcommand\section{\@startsection {section}{1}{\z@}
{-30pt \@plus -1ex \@minus -.2ex}
{2.3ex \@plus.2ex}
{\normalfont\normalsize\bfseries}}

\renewcommand\subsection{\@startsection{subsection}{2}{\z@}
{-3.25ex\@plus -1ex \@minus -.2ex}
{1.5ex \@plus .2ex}
{\normalfont\normalsize\bfseries}}

\renewcommand{\@seccntformat}[1]{\csname the#1\endcsname. }

\makeatother

\newcounter{dummy}
\newtheorem{lemma}[dummy]{Lemma}
\newtheorem{theorem}[dummy]{Theorem}
\newtheorem{corollary}[dummy]{Corollary}

\begin{document}

\begin{center}
\uppercase{\bf On the divisibility of} $a^n \pm b^n$ \uppercase{by powers of} $n$
\vskip 20pt
{\bf Salvatore Tringali\footnote{This research has been funded from the
European Community's 7th Framework Programme (FP7/2007-2013) under Grant
Agreement No. 276487 (project ApProCEM).}}\\
{\smallit Laboratoire Jacques-Louis Lions (LJLL),
Universit\'e Pierre et Marie Curie (UPMC),
4 place Jussieu, bo\^ite courrier 187,
 75252 Paris (cedex 05), France}\\
{\tt tringali@ann.jussieu.fr}\\
\end{center}
\vskip 30pt

\centerline{\bf Abstract}

\noindent
We determine all triples $(a,b,n)$ of positive integers such that $a$ and $b$ are relatively prime and $n^k$ divides $a^n + b^n$ (respectively,  $a^n - b^n$), when $k$ is the maximum of $a$ and $b$ (in fact, we answer a slightly more general question). As a by-product, it is found that, for $m, n \in \mathbb N^+$ with $n \ge 2$, $n^m$ divides $ m^n + 1$ if and only if $(m,n)=(2,3)$ or $(1,2)$, which generalizes problems from the 1990 and 1999 editions of the International Mathematical Olympiad. The results are related to a conjecture by K. Gy\H{o}ry and C. Smyth on the finiteness of the sets $R_k^\pm(a,b) := \{n \in \mathbb N^+: n^k \mid a^n \pm b^n\}$, when $a,b,k$ are fixed integers with $k \ge 3$, $\gcd(a,b) = 1$ and $|ab| \ge 2$.

\pagestyle{myheadings}
\thispagestyle{empty}
\baselineskip=12.875pt
\vskip 30pt

\section{Introduction}

It is a problem from the 1990 edition of the International Mathematical
Olympiad (shortly, IMO) to find all integers $n \ge 2$ such that $n^2 \mid 2^n + 1$. This is reported as Problem 7.1.15 (p. 147) in \cite{Andre09}, together with a solution by the authors (p. 323), which shows that the only possible $n$ is $3$. On another hand, Problem 4 in the 1999 IMO asks for all pairs $(n,p)$ of positive integers such that
$p$ is a (positive rational) prime, $n \le 2p$ and $n^{p-1} \mid (p-1)^n + 1$. This is Problem 5.1.3 (p. 105) in the same book as above, whose solution by the authors (p. 105) is concluded with the remark that ``With a little bit more work, we can even erase the condition $n \le 2p$.'' Specifically, it is found that the required pairs are $(1, p)$, $(2, 2)$ and $(3, 3)$, where $p$ is an arbitrary prime.

It is now fairly natural to ask whether similar conclusions can be drawn in relation to the more general problem of determining all pairs $(m,n)$ of positive integers for which $n^m \mid m^n + 1$.
In fact, the question is answered in the positive, and even in a stronger form, by Theorem \ref{th:main} below, where the following observations are taken into account to rule out from the analysis a few trivial cases: Given $a,b \in \mathbb Z$ and $n, k \in \mathbb N^+$, we have that $1^k \mid a^n \pm b^n$ and $n^k \mid a^n - a^n$. Furthermore, $n^k \mid a^n \pm b^n$ if and only if $n^k \mid b^n \pm a^n$, and also if and only if $n^k \mid (-a)^n \pm (-b)^n$. Finally, $n^k \mid a^n + (- a)^n$ for $n$ odd and $n^k \mid a^n - (-a)^n$ for $n$ even.
\begin{theorem}
\label{th:main}
Let $a,b,n$ be integers such that $n \ge 2$, $a \ge \max(1,|b|)$ and $b \ge 0$ for $n$ even, and set $\delta := \gcd(a,b)$, $\alpha := \delta^{-1} a$ and $\beta := \delta^{-1}b$.
\begin{enumerate}[label={\rm (\roman{*})}]
\item\label{item:th_pt1} Assume that $\beta \ne -\alpha$ when $n$ is odd. Then, $n^a \mid a^n + b^n$ and $n^\alpha \mid \alpha^n + \beta^n$ if and only if $(a,b,n) = (2,1,3)$ or $(2^c, 2^c, 2)$ for $c \in \{0, 1, 2\}$.
\item\label{item:th_pt2}  Assume $\beta \ne \alpha$. Then, $n^a \mid a^n - b^n$ and $n^\alpha \mid \alpha^n - \beta^n$ if and only if $(a,b,n) = (3,1,2)$ or $(2,-1,3)$.
\end{enumerate}
\end{theorem}
Theorem \ref{th:main} will be proved in Section \ref{sec:proofs}. It would be interesting to extend the result, possibly
at the expense of some extra solutions, by removing the assumption that $n^\alpha \mid \alpha^n + \beta^n$ or $n^\alpha \mid \alpha^n - \beta^n$ (we continue with the same notation as in the statement above), but at present we do not have great ideas for this.

On another hand, we remark that three out of the six triples implied by Theorem \ref{th:main} in its current formulation come from the identity $2^3 + 1^3 = 3^2$, while the result yields a solution of the IMO problems which have
originally stimulated this work. Specifically, we have the following corollary (of which we omit the obvious proof):
\begin{corollary}
Let $m,n \in \mathbb N^+$. Then $n^m \mid m^n + 1$ if and only if either $(m,n)=(2,3)$, $(m,n) = (1,2)$, or $n=1$ and $m$ is arbitrary.
\end{corollary}
For the notation and terminology used here but not defined, as well as for
material
concerning classical topics in number theory, the reader should refer to
\cite{Hardy08}. In particular, we write $\mathbb Z$ for the integers, $\mathbb N$ for the nonnegative integers, and $\mathbb N^+$ for $\mathbb N \setminus \{0\}$, each of these sets being endowed with its ordinary addition $+$,
multiplication $\cdot$ and total order $\le$.
For $a,b \in \mathbb Z$ with $a^2
+ b^2 \ne 0$, we denote by $\gcd(a,b)$
the greatest common divisor
of $a$ and $b$.
Lastly, we let $\mathbb P$ be the (positive rational) primes and, for $c \in
\mathbb Z \setminus \{0\}$ and $p \in \mathbb P$, we use $e_p(c)$ for
the greatest exponent $k \in \mathbb N$ such that $p^k \mid c$, which is extended to $\mathbb Z$ by $e_p(0) := \infty$.

We will make use at
some point of the following
lemma, which belongs to
the folklore and is typically attributed to \'E. Lucas \cite{Lu1878} and R. D.
Carmichael \cite{Car1909} (the latter having fixed an error in
Lucas' original work in
the $2$-adic case).
\begin{lemma}[Lifting-the-exponent lemma]
\label{lem:lte}
For all $x,y \in \mathbb Z$, $\ell \in \mathbb N^+$ and $p \in \mathbb P$ such that
$p \nmid xy$ and $p \mid x-y$, the following conditions are satisfied:
\begin{enumerate}[label={\rm (\roman{*})}]
\item If $p \ge 3$, $\ell$ is odd, or $4 \mid x-y$, then $e_p(x^\ell - y^\ell) =
e_p(x-y) + e_p(\ell)$.
\item If $p = 2$, $\ell$ is even and $e_2(x-y) = 1$, then
$e_2(x^\ell - y^\ell) =
e_2(x+y) + e_2(\ell)$.
\end{enumerate}
\end{lemma}
In fact, our proof of Theorem \ref{th:main} is just the result of a
meticulous refinement of the solutions already known for the problems mentioned
in the preamble. Hence, our only possible merit, if any at all, has been that of
bringing into focus a clearer picture of (some of) their essential features.

The study of the congruences $a^n \pm b^n \equiv 0 \bmod n^k$ has a very long
history, dating back at least to
Euler, who proved that, for all relatively prime integers $a,b$ with $a > b
\ge 1$, every primitive prime divisor of $a^n - b^n$ is congruent to $1$ modulo
$n$; see \cite[Theorem I]{Birk1904} for a proof and \cite[\S{}1]{Birk1904} for
the terminology. However, since there are so many results related to the
question, instead of trying to
summarize them here, we just refer the interested reader to the paper \cite{Gyory10}, whose authors
provide an account of the existing literature on the topic and characterize, for $a,b \in \mathbb Z$ and $k \in \mathbb N^+$,
the set $R_k^{+}(a,b)$, respectively $R_k^{-}(a,b)$, of all positive integers $n$ such
that $n^k$ divides $a^n + b^n$, respectively $a^n - b^n$ (note that no assumption is made about the
coprimality of $a$ and $b$), while addressing the problem of finding the exceptional
cases when $R_1^{-}(a,b)$ and $R_2^{-}(a,b)$ are finite; see, in particular,
\cite[Theorems 1--2 and 18]{Gyory10}.
Nevertheless, the related question of determining, given $a,b \in \mathbb Z$
with $\gcd(a,b) = 1$, all positive integers $n$ such that
$n^k$ divides $a^n + b^n$ (respectively, $a^n - b^n$), when $k$ is the maximum of $|a|$ and $|b|$, does
not appear to be considered neither in
\cite{Gyory10} nor in the references therein.

On another hand, it is suggested in \cite{Gyory10} that $R_k^{+}(a,b)$ and $R_k^{-}(a,b)$ are both finite provided that $a,b,k$ are fixed integers with $k \ge 3$, $\gcd(a,b) = 1$ and $|ab| \ge 2$ (the authors point out that the question is probably a difficult one, even assuming the ABC conjecture). Although far from being an answer to this, Theorem \ref{th:main} in the present paper implies that, under the same assumptions as above, $R_k^{+}(a,b)$ and $R_k^{-}(a,b)$ are finite for $k \ge \max(|a|,|b|)$.
\section{Proofs}
\label{sec:proofs}
First, for the sake of exposition, we give a couple of lemmas.
\begin{lemma}
\label{lem:trivial}
Let $x,y,z \in \mathbb Z$ and $\ell \in \mathbb N^+$ such that $\gcd(x,y) = 1$
and $z \mid x^\ell +
y^\ell$. Then $xy$ and $z$ are relatively prime, $q
\nmid x^\ell - y^\ell$ for every integer $q \ge 3$ for which $q
\mid z$, and $4 \nmid z$ provided that $\ell$ is
even. Moreover, if there exists an odd prime divisor $p$ of $z$ and $\ell$ such that
$\gcd(\ell,p-1) = 1$, then $p \mid x +
y$, $\ell$ is odd and $e_p(z) \le e_p(x+y) + e_p(\ell)$.
\end{lemma}
\begin{proof}
The first part is routine (we omit the details). As for the
second, let $p$
be an odd prime dividing both $z$ and $\ell$ with $\gcd(\ell,p-1) = 1$. Also,
considering that $z$ and $xy$ are relatively prime (by the above), denote by
$y^{-1}$ an inverse of $y$ modulo $p$ and by
$\omega$ the order of $xy^{-1}$ modulo
$p$, viz the smallest $k \in
\mathbb N^+$ such that $(xy^{-1})^k \equiv 1 \bmod p$; cf.
\cite[\S{}6.8]{Hardy08}.
Since $(xy^{-1})^{2\ell} \equiv 1 \bmod p$, we have $\omega
\mid 2\ell$.
It
follows from Fermat's little theorem
and \cite[Theorem 88]{Hardy08} that $\omega$ divides $\gcd(2\ell, p-1)$, whence we
get $\omega
\mid 2$, using that $\gcd(\ell,p-1) = 1$. This in turn implies that $p \mid x^2 -
y^2$, to the effect that
either $p \mid x - y$ or $p \mid x + y$. But $p
\mid
x - y$ would give that $p \mid x^\ell - y^\ell$, which is impossible by the
first part of the claim (since $p \ge 3$). So
$p \mid x + y$, with the result that $\ell$ is odd: For, if $2 \mid \ell$, then $p \mid 2x^\ell$ (because $p \mid z \mid x^\ell + y^\ell$ and $y \equiv -x \bmod p$), which would lead to $\gcd(x,y) \ge p$ (again, using that $p$ is odd), viz to a contradiction. The rest is an immediate application of Lemma
\ref{lem:lte}.
\end{proof}
\begin{lemma}
\label{lem:2}
Let $x,y,z \in \mathbb Z$ such that $x,y$ are odd and $x,y \ge 0$. Then $x^2 - y^2 = 2^z$ if and only if $z \ge 3$, $x = 2^{z-2} + 1$ and $y = 2^{z-2} - 1$.
\end{lemma}
\begin{proof}
Since $x$ and $y$ are odd, $x^2 - y^2$ is divisible by $8$, namely $z \ge 3$, and there exist $i,j \in \mathbb N^+$ such that $i+j=z$, $x - y = 2^i$ and $x + y = 2^j$. It follows that $x = 2^{j-1} + 2^{i-1}$ and $y = 2^{j-1} - 2^{i-1}$, and then $j > i$ and $i = 1$ (otherwise $x$ and $y$ would be even). The rest is straightforward.
\end{proof}
Now, we are ready to write down the proof of
the main result.
\begin{proof}[Proof of Theorem \ref{th:main}]
\ref{item:th_pt1} Assume that $n^a \mid a^n + b^n$, $n^\alpha \mid \alpha^n + \beta^n$, and $\beta \ne -\alpha$ when $n$ is odd. Since $\alpha$ and $\beta$ are coprime (by construction), it holds that $\beta \ne 0$, for otherwise $n \mid \alpha^n + \beta^n$ and $n \ge 2$ would give $\gcd(\alpha,\beta) \ge 2$. Also, $\alpha = |\beta|$ if and only if $\alpha = \beta = 1$ and $n = 2$ (as $\beta \ge 0$ for $n$ even), to the effect that $2^\delta$ divides $2\delta^2$, which is possible if and only if $\delta \in \{1, 2, 4\}$ and gives $(a,b,n) = (1,1,2)$, $(2,2,2)$, or $(4,4,2)$. So, we are left with the case when
\begin{equation}
\label{equ:01}
\alpha \ge 2\quad\text{and}\quad \alpha > |\beta| \ge 1,
\end{equation}
since $\alpha \ge \max(1,|\beta|)$. Considering that $4 \mid n^2$ for $n$ even, it follows from Lemma \ref{lem:trivial}
that $n$ is odd and $\gcd(\alpha \beta, n) = 1$. Denote by $p$ the
smallest prime divisor of $n$. Again by Lemma \ref{lem:trivial}, it is then found that $p$ divides $\alpha + \beta$ and
\begin{equation}
\label{equ:1}
\alpha - 1 \le (\alpha - 1) \cdot e_p(n) \le e_p(\alpha + \beta).
\end{equation}
Furthermore, $\alpha + \beta \ge 1$ by equation \eqref{equ:01}, to the effect that
\begin{equation}
\label{equ:2}
\alpha + \beta = p^r s,\quad\text{with }r,s \in \mathbb N^+\text{ and }p \nmid s.
\end{equation}
Therefore, equations \eqref{equ:01} and \eqref{equ:2} yield that $2\alpha \ge p^r s + 1$. This implies by equation \eqref{equ:1}, since $r = e_p(\alpha + \beta)$, that $
3^r s \le p^r s \le 2r + 1$,
which is possible only if $p = 3$ and $r = s = 1$. Thus, by equations \eqref{equ:1} and \eqref{equ:2}, we get $\alpha + \beta = 3$ and $\alpha = 2$, namely $(\alpha,\beta) = (2,1)$. Also, $e_3(n) = 1$, and hence $n = 3t$ for some $t \in \mathbb N^+$ with $\gcd(6,t) = 1$. It follows that $t^2 \mid \gamma^t + 1$ for $\gamma = 2^3$.

So suppose, for the sake of contradiction, that $t \ge 2$ and let $q$ be the least prime divisor of $t$. Then, another application
of Lemma \ref{lem:trivial} gives $2e_q(t) \le e_q(\gamma + 1) + e_q(t)$, and accordingly
$1 \le e_q(t) \le e_q(\gamma + 1) = e_q(3^2)$,
which is however absurd, due to the fact that $\gcd(3,t) = 1$.
Hence $t = 1$, i.e. $n = 3$, and
putting everything together completes the proof, because $2^3 + 1^3 = 3^2$ and $3^{2\delta} \mid \delta^2 \cdot (2^3+1^3)$ only if $\delta = 1$.

\ref{item:th_pt2} Assume that $n^a \mid a^n - b^n$, $n^\alpha \mid \alpha^n - \beta^n$, and $\beta \ne \alpha$. Since $\gcd(\alpha,\beta) = 1$, we get as in the proof of point \ref{item:th_pt1} that $\beta \ne 0$, while $\alpha = |\beta|$ only if $\alpha = 1$, $\beta = -1$, and $n$ is odd (again, $\beta \ge 0$ for $n$ even), which is however impossible, because it would give that $n \mid 2$ with $n \ge 3$.
So, we can suppose from now on that $\alpha$ and $\beta$ satisfy the same conditions as in equation \eqref{equ:01}, and write $n$ as $2^r s$, where $r \in \mathbb N$, $s \in \mathbb N^+$ and $\gcd(2,s) = 1$. We have the following:
\begin{description}
\item[Case 1:] $r = 0$, i.e. $n = s$. Then, $n$ is odd, to the effect that $n^a \mid a^n + (-b)^n$ and $n^\alpha \mid \alpha^n + (-\beta)^n$, so by point \ref{item:th_pt1} we get $(a,b,n) = (2,-1,3)$.

\item[Case 2:] $r \ge 1$. Since $n$ is even and $\gcd(\alpha,\beta) = 1$, both $\alpha$ and $\beta$ are odd, that is $8 \mid \alpha^2 - \beta^2$. It follows from point \ref{item:th_pt1} of Lemma \ref{lem:lte} that
\begin{equation}
\label{equ:added}
e_2(\alpha^n - \beta^n) = e_2(\alpha^2 - \beta^2) + e_2(2^{r-1}s) = e_2(\alpha^2 - \beta^2) + r - 1.
\end{equation}
(With the same notation as in its statement, we apply Lemma \ref{lem:lte} with $x = \alpha^2$, $y = \beta^2$, $\ell = 2^{r-1}s$, and $p = 2$.) Also, $2^{r\alpha} \mid \alpha^n - \beta^n$, so equation \eqref{equ:added} yields
\begin{equation}
\label{equ:5}
(\alpha - 1) \cdot r \le e_2(\alpha^2 - \beta^2) - 1.
\end{equation}
There now exist $u,v \in \mathbb N^+$ with $u \ge 2$ and $\gcd(2,v) = 1$ such that $\alpha^2 - \beta^2 = 2^{u+1} v$, with the result that $\alpha > 2^{u/2} \sqrt{v}$. Hence, taking also into account that $2^x \ge x + 1$ for every $x \in \mathbb R$ with $x \ge 1$, we get by equation \eqref{equ:5} that
\begin{equation}
\left(\frac{u}{2} + 1\right)\sqrt{v} \le 2^{u/2} \sqrt{v} < \frac{u}{r} + 1,
\end{equation}
which is possible only if $r = 1$ and $\sqrt{v} < 2$. Then $2^{u/2}\sqrt{v} < u + 1$, in such a way that $2 \le u \le 5$ and $v = 1$ (using that $v$ is odd).
In consequence of Lemma \ref{lem:2}, all of this implies, at the end of the day, that $\alpha = 2^z + 1$, $b = 2^z - 1$ and $n = 2s$ (recall that we want the conditions in equation \eqref{equ:01} to be satisfied and $\beta \ge 0$ for $n$ even), where $z$ is an integer between $1$ and $4$. But we need $2^z \le z + 1$ by equation \eqref{equ:5}, so necessarily $z = 1$, i.e. $\alpha = 3$ and $\beta = 1$.
Finally, we check that $(2s)^3 \mid 3^{2s} - 1^{2s}$ if and only if $s = 1$: For, if $s \ge 2$ and $q$ is the smallest prime divisor of $s$, then $0 < 3e_q(s) \le e_q(3^2 - 1)$ by Lemma \ref{lem:trivial}, which is absurd since $\gcd(2,s) = 1$. This gives $(a,b,n) = (3,1,2)$, while it is trivially seen that $2^{3\delta} \mid \delta^2 \cdot (3^2 - 1^2)$ if and only if $\delta = 1$.
\end{description}
Putting all the pieces together, the proof is thus complete.
\end{proof}

\section*{Acknowledgements}
The author thanks Paolo Leonetti (Universit\`a Bocconi di Milano) and Carlo Sanna (Universit\`a di Torino) for useful discussions. He is also grateful to an anonymous referee for remarks that have helped to improve the overall quality of the paper.


\begin{thebibliography}{10}

\bibitem{Andre09} T. Andreescu and D. Andrica, \emph{Number Theory - Structures, Examples, and Problems}, 1st ed., Birkh\"auser, 2009.
%
\bibitem{Birk1904} G. D. Birkhoff and H. S. Vandiver, \textit{On the Integral Divisors of $a^n - b^n$}, Ann. of Math. \textbf{5} (Jul., 1904), no. 4, 173--180.
%
\bibitem{Car1909} R. D. Carmichael, \textit{On the Numerical Factors of Certain Arithmetic Forms}, Amer. Math. Monthly \textbf{16} (1909), no. 10, 153--159.
%
\bibitem{Gyory10} K. Gy\H{o}ry and C. Smyth, \textit{The divisibility of $a^n -
b^n$ by powers of $n$}, Integers \textbf{10} (Jul., 2010), no. 3, 319--334.
%
\bibitem{Hardy08} G. H. Hardy and E. M. Wright, \textit{An Introduction to the
Theory of Numbers}, 6th ed. (revised by D. R. Heath-Brown and J. H. Silverman),
Oxford University Press, 2008.
%
\bibitem{Lu1878}  \'E. Lucas, \textit{Th\'eorie des Fonctions Num\'eriques Simplement Periodiques}, Amer. J. Math. \textbf{1} (1878), 184--196, 197-240, 289--321.
%
\end{thebibliography}
\end{document}